\documentclass{amsart}
 \usepackage{ulem}
 \usepackage{setspace}
\usepackage{tabularx}
\usepackage{amsmath,amsfonts,amssymb,xcolor,latexsym,euscript,gensymb}
\usepackage{enumerate}
\usepackage[numbers]{natbib}

\newtheorem{lemma}{\bf Lemma}[section]

\newtheorem{teor}[lemma]{\bf Theorem}

\newtheorem{cor}[lemma]{\bf Corollary}

\newtheorem{problem}[lemma]{\bf Problem}

\input xy
\xyoption{all}

\title[]{On profinite groups with the Magnus Property}

\author{Claude Marion} 
\address{Claude Marion. Centro de Matem\'{a}tica, Faculdade de Ci\^{e}ncias, Universidade do Porto, Rua do Campo Alegre 687, 4169-007 Porto, Portugal
\newline
ORCID: https://orcid.org/0000-0002-3802-1750}
\email{claude.marion@fc.up.pt}

\author{Pavel Zalesskii}
\address{Pavel Zalesskii. Departamento de Matem\'atica, Universidade de Bras\'ilia, Campus 
Universit\'ario Darcy Ribeiro, Bras\'ilia-DF, 70910-900, Brazil. \newline
ORCID: https://orcid.org/0000-0002-2015-239X}
\email{pz@mat.unb.br}

\date{}


\thanks{
The first author acknowledges the support from the Centre of Mathematics of the University of
Porto which is financed by national funds through the Funda\c{c}\~{a}o para a Ci\^{e}ncia e a Tecnologia,
I.P., under the project with references UIDB/00144/2020 and UIDP/00144/2020. The second author acknowledges the support of Capes (Capes-print)}

\begin{document}

\setlength{\parskip}{2mm}

\begin{abstract}
A group is said to have the Magnus Property (MP) if whenever two elements have the same normal closure then they are conjugate or inverse-conjugate.  We  show that a profinite MP group $G$ is prosolvable and any quotient of  it is again MP.  
As corollaries we obtain that the only prime divisors of $|G|$ are $2$, $3$, $5$ and $7$, and the second derived subgroup of $G$ is pronilpotent. We also show that  the inverse limit of an inverse system of profinite MP groups is again MP.  Finally when $G$ is finitely generated, we establish that $G$ must in fact be finite. 
\end{abstract}

\maketitle

\section{Introduction}
Given a group $G$,  elements $x,y\in G$ and  a subset $S\subseteq G$, we let $x^G$ denote the conjugacy class of $x$ in $G$, and $\langle S\rangle$ and $\langle S\rangle^G$ denote respectively the subgroup of $G$ abstractly generated by $S$ and the normal closure of $S$ in $G$  which is the smallest normal subgroup of $G$ containing $S$. Note that $\langle x\rangle^G=\langle x^G\rangle$.

A group $G$ has the \textit{Magnus Property} (MP) if whenever $x,y \in G$ are such that $\langle x^G \rangle = \langle y^G \rangle$, then $x$ is conjugate in $G$ to $y$ or to $y^{-1}$.
We also say that a group $G$ has the \textit{Strong Magnus Property} (SMP) if whenever $x,y \in G$ are such that $\langle x^G \rangle = \langle y^G \rangle$, then $x$ is conjugate in $G$ to $y$. In other words, an SMP group is an MP group in which every element is conjugate to its own inverse. Magnus' original motivation to study MP groups is, as he showed in 1930 in \cite{Magnus30}, that free groups 
have this property. This can also be formulated by saying that if $G$ is any $1$-relator group realized as quotient of the free group $F$, say $G \cong F/\langle x \rangle^F$, then any other relator (on the same generators) realizing the same group must be either conjugate or inverse-conjugate to $x$.

This century, the study of groups having the Magnus Property gained some  further interest. For example, Bogopolski proved in \cite{B} that the fundamental group of  a closed orientable surface is MP.  About a decade later,  Klopsch and Kuckuck showed in \cite{KK} that a direct product of free groups is MP. More recently, Klopsch, Mendon\c{c}a and Petschick characterized in \cite{KMP} the free polynilpotent groups which are MP.  Concerning finite groups,  Garonzi and the first author showed in \cite{GM} that a finite MP group is solvable and there are in fact only eight finite primitive groups that are MP. Moreover, using essentially the latter classification, they established several other properties of a finite MP group $G$:  a prime divisor $p$ of $|G|$ belongs to \{2,3,5,7\} (and if $G$ is nilpotent then a prime divisor of $|G|$ belongs to $\{2,3\}$), $G$ has Fitting height at most 2,  the second derived subgroup of $G$ is nilpotent, and a chief factor of $G$ has order $2$, $3$, $4$, $5$, $7$ or $9$. 

Following the study in \cite{GM} of finite groups which are MP, it is natural to investigate the  profinite groups which are MP. 
We propose to perform this task in this paper. Note first that the profinite and pro-$p$ completions do not preserve Magnus property as was observed in \cite{BZ}.  This indicates that with respect to the Magnus property profinite groups should have similar features as MP finite groups.

By \cite[Proposition 1.1]{GM} a finite MP group is solvable. Our first result establishes the analogous statement for profinite groups. 

\begin{teor}\label{t:profMPpros}
Every profinite MP group is prosolvable. 
\end{teor}

The proof of this theorem is not  a simple projective limit argument, because to perform such argument one needs to know in advance that MP property is preserved by quotients and inverse limits.

In general the Magnus Property is not preserved under taking quotients. Indeed  by \cite{Magnus30} every free group is MP, also every group is a quotient of a free group, however as shown in \cite{GM}, in general a finite group is not MP. 
In \cite[Proposition 2.4]{KMP}, Klopsch, Mendon\c{c}a and Petschick give a sufficient condition for a quotient $G/N$ of a MP group $G$ to be MP, and as a consequence they derive that every quotient of a finite MP group is MP.  This latter observation is a crucial ingredient in showing that every finite MP group is solvable (see \cite[Proposition 1.1]{GM}). However one cannot apply \cite[Proposition 2.4]{KMP} to show that every quotient of a profinite MP group is MP. 
We have to use in fact Theorem \ref{t:profMPpros} and essentially an application of Nakayama's lemma, to establish the result. When considering quotients of a profinite group we mean quotients by a closed normal subgroup. 

\begin{teor}\label{t:quo}
Every quotient of a profinite (S)MP group is (S)MP.
\end{teor}

In particular the proof of Theorem \ref{t:quo} gives another proof to that given in \cite{KMP} for the fact that every quotient of a finite MP group is MP.

Combining Theorem \ref{t:quo} with the work in \cite{GM} on finite groups which are MP, we obtain two further corollaries.

\begin{cor}\label{c:pdivisors}
Let $G$ be a nontrivial profinite MP group and let $p$ be a prime divisor of $|G|$. The following assertions hold.
\begin{enumerate}[(i)]
\item We have $p\in \{2,3,5,7\}$.   
\item If $G$ is SMP or pronilpotent then $p\in \{2,3\}$. 
\item If $G$ is SMP and pronilpotent then $p=2$. 
\end{enumerate}
\end{cor}

Recall that given a profinite group $G$ and a nonnegative integer $n$, the $n$-th derived subgroup $G^{(n)}$ of $G$ is defined recursively by  setting $G^{(0)}$ to be $G$ and $G^{(n+1)}$ to be the closure of the abstract derived subgroup of $G^{(n)}$.

\begin{cor}\label{c:sdpn}
Let $G$ be a profinite MP group. Then the second derived subgroup $G^{(2)}$ of $G$ is pronilpotent. 
\end{cor}

It is natural to ask whether the inverse limit of an inverse system of profinite (S)MP groups is again (S)MP. We answer  this question positively. 
\begin{teor}\label{t:proMP}
The inverse limit of an inverse system of profinite (S)MP groups is (S)MP.
\end{teor}

Recall that a profinite group is finitely generated if it contains a finitely generated abstract group which is dense. We finally establish the following result.

\begin{teor} \label{t:ff}
Every finitely generated profinite MP group is finite.
\end{teor}

Recall that an infinite group is just infinite if each of it proper quotients is finite.  The proof of Theorem \ref{t:ff} uses Theorem \ref{t:quo} and Corollary \ref{c:pdivisors}, as well as the following two  important results. The first one, established by the second author in \cite{Z}, is that an infinite finitely generated profinite group $G$ with only finitely many primes dividing $|G|$ admits just infinite quotients. The second one, established by Jaikin-Zapirain and Nikolov in \cite{JN} is that an infinite compact Hausdorff group, and so an infinite profinite group, has uncountably many conjugacy classes. 

We end the introduction with the following two open problems. Recall that a  group is locally finite if every subgroup which is finitely generated as an abstract group is finite.
\begin{problem}
Is every profinite MP group locally finite?
\end{problem}

\begin{problem}
Let $d$ be a positive integer. Can one bound the order of a finite $d$-generated MP group in terms of $d$? 
\end{problem}

By the solution to the restricted Burnside problem due to Zel'manov  in \cite{Ze1,Ze2}, the above problem is equivalent to determining whether, given  a positive integer $d$, one can bound the exponent of a finite $d$-generated MP group in terms of $d$.

The layout of the article is a follows. 
In Section \ref{s:two} we prove Theorem \ref{t:profMPpros}. In Section \ref{s:three} we prove Theorem \ref{t:quo} and Corollaries \ref{c:pdivisors} and \ref{c:sdpn}. In Section \ref{s:proMP} we prove Theorem \ref{t:proMP}. Finally in Section \ref{s:four} we prove Theorem \ref{t:ff}.

\section{The proof of Theorem \ref{t:profMPpros}}\label{s:two}

In this section we prove Theorem \ref{t:profMPpros}.

\noindent \textit{Proof of Theorem \ref{t:profMPpros}.}
Let $G$ be a profinite MP group. Suppose for a contradiction that $G$ is not prosolvable. Let $G/K$ be a maximal prosolvable quotient of $G$. Since $G$ is not prosolvable, $K$ is nontrivial.  Let $M(K)$ be the intersection of all  maximal closed normal subgroups of $K$. Then $M(K)$ is a closed characteristic subgroup of $K$ and so $M(K)$ is a closed normal subgroup of $G$.  It follows from \cite[Corollary 8.2.3]{RZ} that $K/M(K)$ is a direct product of finite nonabelian simple groups and so $K/M(K)$ is 1-generated as a $G$-group. In fact every element of $K/M(K)$ with nontrivial component on each finite nonabelian simple factor is a generator of $K/M(K)$ as  a  $G$-group.  Moreover any lift to $K$ of a generator of $K/M(K)$ as a $G$-group is a generator of $K$ as a $G$-group (see \cite[Proposition 8.3.6]{RZ}). Since $G$ is an MP group all such lifts are all conjugate or inverse-conjugate in $G$. It follows that all generators of $K/M(K)$ as a $G$-group are conjugate or inverse-conjugate in $G/M(K)$. However since $K/M(K)$ is a direct product of finite nonabelian simple groups, one can find two distinct primes $p$ and $q$ and two such generators $x$ and $y$ of $K/M(K)$ such that their projection onto the first finite nonabelian simple factor $S$ lie respectively in a $p$-Sylow subgroup and a $q$-Sylow subgroup of $S$. In particular $x$ and $y$ are not conjugate nor inverse-conjugate in $G/M(K)$, a contradiction.  $\square$

\section{The proofs of Theorem \ref{t:quo} and its corollaries}\label{s:three}

In this section we prove Theorem \ref{t:quo} and Corollaries \ref{c:pdivisors} and \ref{c:sdpn}. We start with a few results needed for the proof of Theorem \ref{t:quo}. 

\begin{lemma}(Nakayama's lemma)
Let $R$ be a unital associative ring and let $J(R)$ be the Jacobson radical of $R$.  Let $M$ be an $R$-module. The following assertions hold. 
\begin{enumerate}[(i)]
\item If $M$ is  finitely generated as an $R$-module and $J(R)M=M$ then $M=0$.
\item Let $U$ be an $R$-submodule of $M$. If $M/U$ is  finitely generated as an $R$-module and $M=U+J(R)M$, then $U=M$. 
\end{enumerate}
\end{lemma}

Recall that given a ring $R$, an $R$-module $V$ is said to be cyclic if $V=Rv$ for some $v\in V$, and $V$ is said to be homogeneous if $V$ is a direct sum of simple $R$-submodules isomorphic to each other.    

\begin{lemma}\label{l:first}
Let $G$ be a finite group, let $p$ be a prime, and let  $V$ be a semisimple homogeneous cyclic $\mathbb{F}_pG$-module. 
Let $W$ be a quotient module of $V$. Then every generator of $W$
lifts to a generator of $V$.
\end{lemma}

\begin{proof}
Let $\phi$ be the representation of the group algebra $\mathbb{F}_pG$ afforded by $V$
and let $ K=\phi(\mathbb{F}_pG)$. Since $V$ is semisimple, it follows from the Artin-Wedderburn theorem that $K$ is a direct product of matrix algebras of the form $M_n(\mathbb{F}_{p^\ell})$ for some positive integers $n,\ell$. As $V$ is homogeneous, it follows that $K=M_n(\mathbb{F}_{p^\ell})$, and one can view $V$ as a homogeneous $K$-module.
Clearly, $W$ is also homogeneous and cyclic as a $K$-module.

We can write $V=V_1\oplus V_2$, where $V_1\cong W$ and $V_2$ is some complement. Since $V$ is cyclic as a $K$-module, we can  write $V=Kv$ for some $v\in V$. Write also  $v=v_1+v_2$, where $v_i\in V_i$ for $i\in \{1,2\}$. Then $Kv_i=V_i$ (for $i\in \{1,2\}$)
as $Kv=V$.

Let $v_1'$ be any generator of $V_1$. Then $kv_1=v_1'$ for some invertible element $k\in K$. Moreover $kv=v_1'+kv_2$ is a generator of $V$ as a $K$-module, since $k$ is invertible. Hence $v_1'$ lifts to a generator $kv$ of $V$ as required.
\end{proof}

\begin{lemma}\label{l:second}
Let $G$ be a group and  let $p$ be a prime. Suppose that $f: M\rightarrow N$ is an epimorphism of nonzero  finite cyclic  $\mathbb{F}_pG$-modules. Then a generator of $N$ lifts to a generator of $M$. 
\end{lemma}

\begin{proof}
Write $R=\mathbb{F}_pG$ and let $J=J(R)$ be the Jacobson radical of $R$.
We obtain an epimorphism $\phi: M/JM\rightarrow N/JN$ of  finite nonzero cyclic $R$-modules. 
Now $M/JM$ and $N/JN$ are semisimple. Set $V=M/JM$.  Then $V=W\oplus U$ where $W\cong N/JN$ and $U$ is some complement.   Let $\psi:V\rightarrow W$ be the canonical epimorphism. 
Write $$V=\bigoplus_{\alpha\in A} V_\alpha$$ and  $$W=\bigoplus_{\beta\in B}W_\beta$$ into  a direct sum of non-isomorphic homogeneous components. Note that the sets $A$ and $B$ are finite. Since $V$ and $W$ are cyclic, so are $V_\alpha$ and $W_\beta$ for every $\alpha\in A$ and every $\beta\in B$. \\ 
Let $w$ be any  generator of $W$. Write $w=\sum_{\beta\in B}w_\beta$ where $w_\beta\in W_\beta$ for every $\beta\in B$.  Since $W=Rw$, $W=\oplus_{\beta\in B}W_\beta$ and $W_\beta$ are all cyclic, it follows that $W_\beta=Rw_\beta$ for every $\beta\in B$. \\
 Fix $\beta \in B$. There exists $\alpha_\beta \in A$ such that  $\psi:V_{\alpha_\beta}\rightarrow W_\beta$ is an epimorphism. 
Since $V_{\alpha_\beta}$ and $W_\beta$ are finite, it follows from Lemma \ref{l:first} that there exists $v_{\alpha_\beta} \in V_{\alpha_\beta}$ such that  $v_{\alpha_\beta}$ generates $V_{\alpha_\beta}$ and $\psi(v_{\alpha_\beta})=w_\beta$. \\
Let 
$$A_W=\{\alpha \in A: \psi(V_\alpha)\neq0\}=\{\alpha\in A: \psi(V_\alpha)=W_\beta \ \textrm{for some} \ \beta \in B\}=\{\alpha_\beta\in A: \beta\in B\}$$ and $$A_U=\{\alpha \in A: \psi(V_{\alpha})= 0\}.$$
For every $\delta \in A_U$ choose a generator $v_\delta$ of $V_\delta$. Let $v=\sum_{\alpha_{\beta} \in A_W} v_{\alpha_\beta}+\sum_{\delta \in A_U} v_\delta$.   Then $V=Rv$ and $\psi(v)=w$. \\
In particular if $n+JN$ is any generator of $N/JN$ then there is an element $m+JM$ of $M/JM$ generating $M/JM$ and such that $\phi(m+JM)=f(m)
+JN=n+JN$.  It follows that $f(m)=n+n_1$ for some $n_1\in JN$. As $f: M\rightarrow N$ is surjective, $f(JM)=JN$ and there is an element $m'\in JM$ such that $f(m')=-n_1$. In particular $f(m+m')=n$ and $m+JM=m+m'+JM$. Now $M=\langle m+m'\rangle+JM$ and by Nakayama's lemma, it follows that $M=\langle m+m'\rangle$ with $f(m+m')=n$. 
\end{proof}

\begin{cor}\label{c:modules}
Let $G$ be a profinite group and let $p$ be a prime. Suppose that $f: M\rightarrow N$ is an epimorphism of nonzero  cyclic  $\mathbb{F}_pG$-modules where $M$ and $N$ are second countable. Then a generator of $N$ lifts to a generator of $M$. 
\end{cor}

\begin{proof}
Write $R=\mathbb{F}_pG$. Let $n$ be any generator of $N$. We first consider the case where $N$ is finite.  Write $$M=\varprojlim_{i\in \mathbb{N}} M_i$$ where each $M_i$ is a finite cyclic quotient $\mathbb{F}_pG$-module of $M$, let $\{M_i,\phi_{ij}\}$ be the corresponding (surjective) inverse system and for each $i\in \mathbb{N}$ let $\phi_i: M \rightarrow M_i$ be the projection map.\\
Since $N$ is finite, it now follows from \cite[Proposition 1.16(b)]{RZ} that $f$ factors through some $\phi_k$. In other words there exist some $k\in \mathbb{N}$ and a continuous map $f':M_k\rightarrow N$ such that $f=f'\phi_k$.  We have the following commutative diagram
\[ 
\xymatrix{
M  \ar[r]^{\phi_k} \ar[dr]_f & M_k \ar[d]^{f'}\\
& N
}
\]
 By Lemma \ref{l:second}, we can lift $n$ to a generator $m_k$ of $M_k$, and, since the inverse system is surjective, in turn to a generator $m_{k+1}$ of $M_{k+1}$ with $\phi_{k+1,k}(m_{k+1})=m_k$, and so on. For $1 \leq j\leq k-1$, let $m_j=\phi_{kj}(m_k)$. Then $m_j$ is a generator of $M_j$ for $1\leq j \leq k-1$ and the element $(m_i)_{i\in \mathbb{N}}$ is the required lift of $n$ to a generator of $M$.\\
 
 Suppose finally that $$N=\varprojlim_{i\in \mathbb{N}} N_i$$ where each $N_i$ is finite. 
 For each $i$, let $n_i$ be the projection of $n$ onto $N_i$, and let $S_i$ be the set of all liftings of $n_i$ to $M$. Each set $S_i$ is closed and $S_j\subseteq S_i$ for $j \geq i$. Set $S=\cap_{i \in \mathbb{N}} S_i$. As every $S_i$ is closed, $S\neq \emptyset.$ An element $s\in S$ is a needed lift of $n$ to a generator of $M$. 
\end{proof}

We can now prove Theorem \ref{t:quo}

\noindent \textit{Proof of Theorem \ref{t:quo}.}
Let $G$ be a profinite (S)MP group and let $U$ be a closed normal subgroup of $G$. We need to show that $G/U$ is MP. Write $\overline{G}=G/U$.  By Theorem \ref{t:profMPpros} $G$ is prosolvable and so $\overline{G}$ is also prosolvable.  Consider two elements $\overline{g}$ and $\overline{h}$
of $\overline{G}$ where $g,h\in G$ and $\langle \overline{g} \rangle^{\overline{G}}=\langle \overline{h}\rangle^{\overline{G}}$. Without loss of generality, $\langle \overline{g} \rangle^{\overline{G}}$ is nontrivial. Let $M(\langle g\rangle^G)$ and  $M(\langle \overline{h} \rangle^{\overline{G}})$  denote respectively the intersection of all  maximal closed normal subgroups of  $\langle g\rangle^G$ and $\langle \overline{h}\rangle^{\overline{G}}$. Consider the two prosolvable groups $\frac{\langle g\rangle^{G}}{M(\langle g \rangle^{G})}$ and $\frac{\langle \overline{h}\rangle^{\overline{G}}}{M(\langle \overline{h} \rangle^{\overline{G}})}$.  We have an epimorphism $$\frac{\langle g\rangle^{G}}{M(\langle g \rangle^{G})}\rightarrow \frac{\langle \overline{h}\rangle^{\overline{G}}}{M(\langle \overline{h} \rangle^{\overline{G}})}$$ of groups, and they 
respectively decompose as a direct product of simple modules, where each simple module is a $\mathbb{F}_pG$-module for some prime $p$. It follows from Corollary \ref{c:modules}  that without loss of generality $\langle h\rangle^GM(\langle g \rangle^G)=\langle g\rangle^G$ and so $\langle h\rangle^G=\langle g\rangle^G$. As $G$ is MP, $g$ and $h$ are conjugate or inverse-conjugate in $G$, moreover, when $G$ is SMP, $g$ and $h$ are conjugate in $G$. It follows that $\overline{g}$ and $\overline{h}$ are conjugate or inverse-conjugate in $\overline{G}$, moreover, when $G$ is SMP, $\overline{g}$ and $\overline{h}$ are conjugate in $\overline{G}$. This establishes the result. 
$\square$

We can now prove Corollaries \ref{c:pdivisors} and \ref{c:sdpn}.

\noindent \textit{Proof of Corollary \ref{c:pdivisors}.}
By Theorem \ref{t:quo} every quotient of  a profinite (S)MP group is (S)MP.  Also by \cite[Corollary 1.6]{GM} a prime divisor of the order of a finite MP group belongs to $\{2,3,5,7\}$, a prime divisor of the order of a finite SMP group or a finite nilpotent MP  belongs to $\{2,3\}$, and finally a prime divisor of the order of a finite nilpotent SMP group must be equal to 2. The result follows from the fact that a profinite group is the inverse limit of a surjective inverse system of finite groups. $\square$

\noindent \textit{Proof of Corollary \ref{c:sdpn}.}
By Theorem \ref{t:quo} every quotient of  a profinite MP group is MP.  Also by \cite{GM}  the second derived subgroup of a finite MP group is nilpotent. Since any epimorphism of groups sends the second derived subgroup onto the second derived subgroup,  the result follows. $\square$

\section{The proof of Theorem \ref{t:proMP}}\label{s:proMP}

In this section we prove Theorem \ref{t:proMP}.

\noindent \textit{Proof of Theorem \ref{t:proMP}.}
We first treat the case   $$G=\varprojlim_{i\in I} G_i$$ where each $G_i$ is a finite (S)MP group. Let $\{G_i,\phi_{ij}\}$ be a corresponding inverse system and for each $i\in I$, let $\phi_i: G \rightarrow G_i$ be the projection map. \\
We claim that without loss of generality we can assume that the latter inverse system is surjective.   For each $i\in I$, let $G_i'=\phi_i(G)$. For all $i,j\in I$ with $i\succcurlyeq j$, $\phi_{ij}\phi_i=\phi_j$. Therefore the  maps $\phi_{ij}: G_i \rightarrow G_j$ restrict to maps $\phi_{ij}': G_i'\rightarrow G_j'$, and we obtain the inverse system $\{G_i',\phi_{ij}'\}$ where the transition maps $\phi_{ij}'$ are surjective. By construction 
$$ G=\varprojlim_{i\in I} G_i =\varprojlim_{i\in I} G_i'.$$
It remains to check that for each $i\in I$, $G_i'$ is (S)MP.  Fix $i \in I$ and consider the epimorphism $\phi_i: G \rightarrow  G_i'$. By Lemma \cite[Lemma 1.1.16(a)]{RZ} there exists some $k\in I$ and some mapping $\rho_i': G_k\rightarrow G_i'$ such that $\phi_i=\rho_i'\phi_k$. In particular $G_i'=\rho_i'(G_k)$ and as $G_k$ is a finite (S)MP group, so is $G_i'$ (by Theorem \ref{t:quo}). This establishes the claim. \\
 Let $x,y \in G$ be such that $\langle x\rangle^G=\langle y\rangle^G$. 
Fix $i\in I$. Note that $\phi_i: G \rightarrow G_i$ is surjective and $\langle \phi_i(x)\rangle^{G_i}=\langle \phi_i(y)\rangle^{G_i}$. Since $G_i$ is MP, it follows that $\phi_i(x)$ and $\phi_i(y)$ are conjugate or inverse-conjugate in $G_i$, moreover if $G_i$ is SMP then  $\phi_i(x)$ and $\phi_i(y)$ are conjugate in $G_i$. Hence $x$ and $y$ are conjugate or inverse-conjugate in $G$, moreover if $G$ is a pro-SMP group then $x$ and $y$ are conjugate in $G$. The result follows in this case.

Now we consider the general case. Let $$G=\varprojlim_{i\in I} G_i$$ be an inverse system of profinite (S)MP groups. Decompose each $G_i$ as an inverse limit $G_i=\varprojlim_{i_j\in J_i} G_{i_j}$ of finite (S)MP-groups. By the above it suffices to show that $\{G_{i_j}\mid i\in I, i_j\in J_i\}$ forms an inverse system. To check this, we have to complete any diagram 
$$\xymatrix{&G_{i_j}\ar[d]\\
 G_{\ell_k}\ar[r]& G_{s_r}}$$ to a commutative diagram $$\xymatrix{G_{n_m}\ar[r]\ar[d]&G_{i_j}\ar[d]\\
 G_{{\ell}_k}\ar[r]& G_{s_r}}$$ 
 
 Choose $n\succcurlyeq i,\ell$. Then we have a commutative diagram 
 $$\xymatrix{G_n\ar[r]^{f_{i_j}}\ar[d]^{f_{\ell_k}}&G_{i_j}\ar[d]\\
 G_{{\ell}_k}\ar[r]& G_{s_r}}.$$ 
 Put $N=\mathrm{Ker}(f_{i_j})\cap \mathrm{Ker}(f_{\ell_k})$. Then $G_n/N$ is a finite (S)MP group and we have the following commutative diagram
 
 $$\xymatrix{G_n/N\ar[r]^{f_{i_j}}\ar[d]^{f_{{\ell}_k}}&G_{i_j}\ar[d]\\
 G_{{\ell}_k}\ar[r]& G_{s_r}}$$ as required.
$\square$

\section{The proof of Theorem \ref{t:ff}}\label{s:four}

In this section we prove Theorem \ref{t:ff}. 

\noindent \textit{Proof of Theorem \ref{t:ff}.} Let $G$ be a finitely generated profinite MP group. Suppose for a contradiction that $G$ is infinite. 
Recall that an infinite group is said to be just infinite if each of its proper quotients is finite. By \cite[Theorem 13]{Z}  an infinite finitely generated profinite group with order divisible by only finitely many primes admits just infinite quotients. Hence by Corollary \ref{c:pdivisors} $G$ admits  a just infinite quotient, say $G_J$. Note that  $G_J$ admits only countably many normal subgroups. Also by Theorem \ref{t:quo} $G_J$ is MP.   By \cite[Theorem 1.1]{JN}, every infinite profinite group has uncountably many conjugacy classes. Since $G_J$ is MP and admits only countably many normal subgroups, it now follows that $G_J$ is finite, a contradiction. 
$\square$

\medskip
\noindent {\it Acknowledgements}. The second author thanks Peter Symonds and  Alexandre Zalesski for useful discussions that led to the proofs of Lemmas \ref{l:first} and \ref{l:second}.


\begin{thebibliography}{99}

\bibitem{BZ} M. Boggi, P.A. Zalesskii, A restricted Magnus property for profinite surface groups, Transactions of the American  Mathematical  Society, v. 371, p. 1, 2018.

\bibitem{B} O. Bogopolski.  A surface groups analogue of a theorem of Magnus. Geometric methods in group theory. Contemp. Math. \textbf{372}
American Mathematical Society, Providence, RI, 2005.

\bibitem{GM} M. Garonzi, C. Marion. On finite groups with the Magnus Property. Bulletin of the London Mathematical Society.  Published online on July 09, 2024. 

\bibitem{JN}  A. Jaikin-Zapirain, N. Nikolov. An infinite compact Hausdorff group has uncountably many conjugacy classes.
Proceedings of the American  Mathematical  Society \textbf{147} (2019),  4083--4089.


\bibitem{KK} B. Klopsh, B. Kuckuck. The Magnus property for direct products. Archiv der Mathematik \textbf{107} (2016), 379--388.

\bibitem{KMP} B. Klopsch, L. Mendon\c{c}a, J. M. Petschick. Free polynilpotent groups and the Magnus property. Forum Mathematicum \textbf{35} (2023), 573–590.


\bibitem{Magnus30} W. Magnus.  \"{U}ber diskontinuierliche Gruppen mit einer definierenden Relation. (Der Freiheitssatz).
Journal f\"{u}r die Reine und Angewandte Mathematik. [Crelle's Journal] \textbf{163} (1930), 141--165.

\bibitem{RZ} L. Ribes, P. Zalesskii. Profinite groups. Springer-Verlag, Berlin, 2000, xiv+435 pp.

\bibitem{Z} P. Zalesskii.  Profinite groups admitting just infinite quotients. Monatshefte f\"{u}r Mathematik \textbf{135} (2002), 167--171.

\bibitem{Ze1} E. Zel'manov. 
Solution of the restricted Burnside problem for groups of odd exponent. Izvestiya Akademii Nauk SSSR. Seriya Matematicheskaya \textbf{54} (1990), 42--59. 


\bibitem{Ze2} E. Zel'manov. 
 Solution of the restricted Burnside problem for  $2$-groups. Matematicheskii Sbornik
 \textbf{182} (1991), 568--592.


\end{thebibliography}
\end{document}